\documentclass[9pt,journal,oneside,web]{ieeecolor}
\usepackage{generic}
\usepackage{cite}
\usepackage{graphicx}
\usepackage{graphics}
\usepackage{textcomp}
\usepackage{epsfig}
\usepackage{times}
\usepackage{amsmath}
\usepackage{amssymb}
\usepackage{amsfonts}
\usepackage{float}
\usepackage{siunitx}
\usepackage{scalerel}
\usepackage{cite}
\usepackage{textcomp}
\usepackage{algorithm}
\usepackage{algorithmic}
\usepackage{stackengine}
\usepackage{mathtools}
\usepackage{caption}
\usepackage{subcaption}
\usepackage{nopageno}
\usepackage{comment}
\def\BibTeX{{\rm B\kern-.05em{\sc i\kern-.025em b}\kern-.08em
    T\kern-.1667em\lower.7ex\hbox{E}\kern-.125emX}}

\newtheorem{theorem}{Theorem}
\newtheorem{proposition}{Proposition}
\newtheorem{lemma}{Lemma}

\newtheorem{remark}{Remark}
 \newtheorem{assumption}{Assumption}

\title{\LARGE \bf Distributed Estimation}
\newcommand{\ncom}{\newcommand}

\newcommand{\beqn}{\begin{eqnarray*}}
\newcommand{\eeqn}{\end{eqnarray*}}
\newcommand{\beq}{\begin{eqnarray}}
\newcommand{\eeq}{\end{eqnarray}}

\newcommand{\norm}[1]{\left\lVert #1 \right\rVert}
\newcommand{\inprod}[2]{\left\langle #1, #2 \right\rangle}
\ncom\R{\mathbb{R}}

\author{Anik Kumar Paul$^{1}$, Karthik Shenoy$^{2}$, Arun D. Mahindrakar$^{3}$
\thanks{Anik  is a Wallmart Postdoc at  IISc, Bengaluru, India {email: anikpaul42@gmail.com}}. \thanks{Karthik is a doctoral student in the Department of Electrical Engineering, IIT Madras, Chennai-600036, India, {email: ee21d405@smail.iitm.ac.in}
        }
        \thanks{Arun is with the Department of Electrical Engineering, Indian Institute of Technology Madras, Chennai-600036, India (email: arun\_dm@iitm.ac.in) }
    }

\begin{document}
\title{Stochastic Recursive Inclusions under Biased Perturbations: An Input-to-State Stability Perspective}
\maketitle
\thispagestyle{empty}
\pagestyle{empty}
\begin{abstract}
This paper investigates the asymptotic behavior of stochastic recursive inclusions in the presence of non-zero, non-diminishing bias, a setting that frequently arises in zeroth-order optimization, stochastic approximation with iterate-dependent noise, and distributed learning with adversarial agents. The analysis is conducted through the lens of input-to-state stability of an associated differential inclusion, which serves as the continuous-time limit of the discrete recursion.
We first establish that if the limiting differential inclusion is input-to-state stable and the iterates remain almost surely bounded, then the iterates converge almost surely to the neighborhood of desired equilibrium. We then provide a verifiable sufficient condition for almost sure boundedness by assuming that the underlying operator is single-valued and globally Lipschitz. Finally, we show that several zeroth-order variants of stochastic gradient  naturally fit within this framework, and we demonstrate their input-to-state stability under standard conditions. Overall, the results provide a unified theoretical foundation for studying almost sure convergence of biased stochastic approximation schemes through the Input to State stability theory of differential inclusions.
\end{abstract}
\section{Introduction}
In this paper, we study the stochastic recursive inclusion  
\begin{equation}
\begin{split}
x_{n+1} - x_n - \alpha_n M_{n+1} 
&\in \alpha_n G(x_n) \\
&\in \alpha_n\big(H(x_n) + \Bar{B}(0,\epsilon)\big),
\end{split}
\label{eq:sri}
\end{equation}
where $\{x_n\}$ denotes the iterates, $\{\alpha_n\}$ is a step-size sequence, $\{M_{n+1}\}\subset \mathbb{R}^d$ is a martingale difference sequence with respect to standard filtration, $H:\mathbb{R}^d \rightrightarrows \mathbb{R}^d$ is a set-valued map, and $\Bar{B}(0,\epsilon)$ is the closed ball of radius $\epsilon$ centered at the origin.

Recursions of the form \eqref{eq:sri} appear widely in stochastic approximation, especially in problems where one seeks a point $x^\ast$ satisfying $0 \in H(x^\ast)$ \cite{borkar2008stochastic}. A prominent example is the stochastic subgradient method, where $H(x)$ corresponds to the subdifferential, $M_{n+1}$ models zero-mean noise, and $\epsilon$ represents a non-zero bias. Similar formulations arise across several learning and estimation settings, including Q-learning \cite{watkins1992q}, neural network--based system identification \cite{haykin2004comprehensive,benaim1995global,pham1995neural}, adaptive signal processing \cite{haykin2002adaptive,298540}, and parameter tracking in linear dynamical systems with noisy measurements \cite{guo1995performance,solo1994adaptive}.

The asymptotic behavior of \eqref{eq:sri} has been studied in \cite{benaim2005stochastic}. Under the assumption that the sequence ${x_n}$ remains bounded almost surely, and under standard conditions such as the upper semi-continuity of the set-valued map and the Robbins--Monro requirements on the step-size, the asymptotic behavior of the iterates can be inferred by examining the stability properties of the differential inclusion
\begin{equation}
\dot{x}(t) \in H(x(t)) + \Bar{B}(0,\epsilon).
\label{ODOS}
\end{equation}
 More over in the unbiased case $\epsilon = 0$, classical results show that if $x^\ast \in \mathbb{R}^d$ is an asymptotically stable equilibrium of \eqref{ODOS}, then $x_n \to x^\ast$ almost surely \cite{borkar2008stochastic,kushner2003stochastic}.

However, in many practical settings---particularly in zeroth-order optimization---one does not have direct access to a subgradient at each iteration. Instead, the subgradient must be approximated using noisy function evaluations. Such an approximation is generally biased, and the resulting recursion naturally takes the form of \eqref{eq:sri} (see \cite{prashanth2025gradient}). Non-zero bias also arises in several other stochastic approximation problems, especially when the noise across iterations is not independent and depends on the current iterate \cite{karimi2019non,wang2019multistep}, or in distributed learning scenarios where a subset of agents may behave adversarially, thereby introducing systematic bias \cite{alistarh2018byzantine,data2021byzantine}.  
The analysis in this paper is carried out from the standpoint of zeroth-order optimization, although the results extend to the broader classes of applications mentioned above.

We list the main contributions of this paper below (without loss of generality, we assume that $0 \in H(0)$).

\begin{itemize}

    \item \textbf{Almost Sure Convergence.}
    Our first contribution establishes that if the differential inclusion in \eqref{ODOS} is input-to-state stable and the iterates generated by \eqref{eq:sri} remain almost surely bounded, then the sequence $x_n$ converges to a neighborhood of $0$ almost surely (Theorem~\ref{thm:eps_convergence}). Earlier analyses of \eqref{eq:sri}---for example, \cite{bertsekas2000gradient,vidyasagar2024convergence,karandikar2024convergence,karandikar2025revisiting}---typically rely on the assumption that the bias term vanishes asymptotically. However, a diminishing bias can slow down learning in practice, especially in machine learning implementations (see Chapter~4.4 of \cite{haykin2009neural}). Our numerical simulations also reflect this behavior; see Fig ~\ref{fig: muvarwithPL} and~\ref{fig: muvarwithoutPL}. Moreover, such an assumption can be restrictive in many practical settings. In particular,
in stochastic zeroth-order optimization with biased oracles—common in reinforcement
learning and statistical learning—one typically encounters a nonzero and
non-diminishing bias \cite{bhavsar2022nonasymptotic}. Similarly, in distributed
learning systems, persistent bias may arise due to adversarial agents and is often
beyond the designer’s control, rendering diminishing-bias assumptions unrealistic
in many applications \cite{data2021byzantine}.

\item \textbf{Almost Sure Boundedness.}
A central assumption in stochastic approximation is the almost sure boundedness of the iterates \cite{metivier1984applications}, yet verifying this assumption is often nontrivial. In this paper, we provide a sufficient condition ensuring boundedness: when $H$ is single-valued and globally Lipschitz, the iterates generated by \eqref{eq:sri} remain almost surely bounded (Theorem~\ref{thm:boundedness}). Existing results establish almost sure boundedness under the assumption of a
nonzero but diminishing bias
\cite{vidyasagar2024convergence,karandikar2025revisiting,karandikar2024convergence}.   Boundedness results for gradient descent schemes under global Lipschitz continuity of the gradient have also been reported in \cite{ramaswamy2017analysis}; however, those works do not account for stochastic martingale difference noise, which limits their applicability in many practical settings.

In contrast, this work provides sufficient conditions, together with rigorous justification, under which the recursion \eqref{eq:sri} remains almost surely bounded even in the presence of a persistent bias.

\item \textbf{Stochastic Zeroth-Order Optimization.}
We also analyze stochastic gradient descent with a non-diminishing biased gradient oracle, motivated by zeroth-order optimization settings \cite{allen2018make}. We show that the resulting iterates can be written in the general stochastic recursion form~\eqref{eq:sri}. Under standard conditions, including the Polyak--\L{}ojasiewicz (PL) inequality, we establish that the iterates are almost surely bounded and that the associated dynamical systems for these stochastic gradient methods are input-to-state stable.

We then extend the analysis to convex, nonsmooth, and constrained optimization problems. In this case, when the subgradient oracle is biased, the iterates again admit the representation~\eqref{eq:sri}, with the set-valued map \(H(x)\) given by the Minkowski sum of the subdifferential of the objective function and the truncated normal cone of the constraint set (see~\eqref{PZSGDISS}). Assuming strong monotonicity of the subdifferential mapping, we show that the corresponding dynamical system is input-to-state stable (Proposition~\ref{ISS4}).

Overall, this section highlights the broad applicability of our main results and
demonstrates that input-to-state stability provides a unified and powerful
framework for analyzing the asymptotic behavior of zeroth-order stochastic
gradient methods across smooth, nonsmooth, and constrained settings.

\end{itemize}

\section{Notation and Math Preliminaries}

Let $\mathbb{I}_n$ denote the $n\times n$ identity matrix. The index set $\{1,2,\cdots,N\}$ will be compactly denoted as $[N]$. 
We denote the Euclidean norm as $\|x\|$ for $x\in\mathbb{R}^d$. The $\delta$-neighborhood of a set $C\subset \mathbb{R}^d$ is defined as $N_\delta(C)=\{x\in\mathbb{R}^d\;|\; \inf_{y\in C}\|x-y\|<\delta \}$. 
For $X,Y\subseteq \mathbb{R}^d$, $F: X\rightrightarrows Y$ denotes a set-valued map.
The gradient and Hessian of $f:\mathbb{R}^d\to\mathbb{R}$ will be represented by $\nabla f$ and $\nabla^2f$ respectively. For a convex function $f:\mathbb{R}^d\to\mathbb{R}$ the subdifferential of $f$ at $x\in\mathrm{dom}(f)$ is denoted by $\partial f(x)$. Let $x\sim\mathcal{N}(\mu,\sigma^2)$ denote a random variable which is normally distributed with mean $\mu$ and standard deviation $\sigma$. The expectation of a random variable $X$ will be denoted by $\mathbb{E}[X]$ The sigma algebra generated by the random vectors $X_1,X_2,\ldots, X_n$ are denoted by $\sigma(X_1,X_2,\ldots,X_n)$.

    The normal cone to a non-empty, closed, convex $C\subset\mathbb{R}^d$ at $x\in C$ is defined as $\mathcal{N}_C(x)=\{v\in\mathbb{R}^d\;|\;v^\top(y-x)\leq0,\;\forall\;y\in C\}$. Similarly, the tangent (contingent) cone to $C$ at $x$ is defined as $\mathcal{T}_C(x)=\{w\in\mathbb{R}^d\;|\;\exists\;t_i\downarrow0\; \mathrm{and}\;w_i\to w,\mathrm{such\;that}\;x+t_iw_i\in C\}$

\section{Asymptotic Convergence Analysis}
We first list down all the main  assumptions for this section.

\begin{assumption}
There exists a continuously differentiable function $V: \mathbb{R}^d \to \mathbb{R}$ satisfying $V(x) \geq 0$, $V(0) = 0$ and $\lim_{\norm{x}\to\infty} V(x) = \infty$, along with two class $\mathcal{K}_\infty$ functions $a(\cdot)$ and $b(\cdot)$ such that, for all $x \in \mathbb{R}^d$, $\nu \in H(x)$, and $b \in \bar{B}(0,\epsilon)$,
\begin{equation*}
\langle \nabla V(x), \nu + b \rangle \leq -a(\norm{x}) + b(\epsilon).
\end{equation*}
\label{ISS}
\end{assumption}

\begin{assumption}
The step-size sequence ${\alpha_n}$ satisfies
\begin{equation*}
\sum_{n=1}^{\infty} \alpha_n = \infty, \qquad
\sum_{n=1}^{\infty} \alpha_n^2 < \infty.
\end{equation*}
\label{step-size}
\end{assumption}

\begin{assumption}
The set-valued map $H: \mathbb{R}^d \rightrightarrows \mathbb{R}^d$ is Marchaud \cite{borkar2008stochastic}, that is,
\begin{enumerate}
\item for each $x \in \mathbb{R}^d$, $H(x)$ is nonempty, compact, and convex
\item there exists $\kappa > 0$ such that
\begin{equation*}
\sup_{\nu \in H(x)} \norm{\nu} \leq \kappa(1 + \norm{x}),
\end{equation*}
\item The graph of $H$ is closed.
\end{enumerate}
\label{Mar}
\end{assumption}

\begin{assumption}
Let ${\mathcal{F}_n}$ denote the natural filtration generated by ${x_k}$, that is,
\begin{equation*}
\mathcal{F}_n = \sigma({x_k : 1 \leq k \leq n}).
\end{equation*}
The noise sequence ${M_{n+1}}$ satisfies
\begin{equation*}
\mathbb{E}[M_{n+1} | \mathcal{F}_n] = 0, \qquad
\mathbb{E}[\norm{M_{n+1}}^2 | \mathcal{F}_n] \leq K,
\end{equation*}
for some constant $K > 0$.
\label{noise}
\end{assumption}
The primary objective of the first part of this paper is to demonstrate that the asymptotic behavior of the stochastic recursive inclusion \eqref{eq:sri} can be analyzed through the lens of Input-to-State Stability (ISS) of the associated differential inclusion
\begin{equation}
\dot{x}(t) \in H(x(t)) + \bar{B}(0,\epsilon).
\label{odi}
\end{equation}

Assumption~\ref{ISS} ensures that the perturbed system is input-to-state stable; that is, there exist functions $\beta \in \mathcal{KL}$ and $\gamma \in \mathcal{K}_\infty$ such that, for any initial condition $x(0)$ and admissible input $b(t) \in \bar{B}(0,\epsilon)$,
\begin{equation*}
\norm{x(t)} \leq \beta(\norm{x(0)}, t) + \gamma(\epsilon), \qquad \forall \;  t \geq 0.
\end{equation*}

Assumption~\ref{step-size} corresponds to the standard Robbins--Monro conditions on the step-size sequence, ensuring diminishing yet persistent updates. 
Assumption \ref{Mar} is a standard requirement in the stochastic approximation literature \cite{benaim2005stochastic,borkar2008stochastic}. This assumption further ensures that the associated set-valued map is upper semi-continuous (see \cite{aubin1993differential} for definition), which in turn guarantees the existence of a Carathéodory solution to the differential inclusion \eqref{odi}. 
Finally, Assumption~\ref{noise} characterizes the stochastic noise sequence ${M_{n}}$ as a martingale difference sequence with uniformly bounded second moments, which facilitates the connection between the discrete-time recursion and its continuous-time limiting dynamics.

Before proceeding with the asymptotic analysis of the iterates ${x_n}$ via the corresponding differential inclusion, we first introduce their continuous-time interpolated trajectory. Specifically, define $\Bar{X}(t)$ as
\begin{align}
\Bar{X}(t) = x_n + \big(x_{n+1} - x_n\big) \frac{t - t(n)}{t(n+1) - t(n)},\; \; t \in I_n,
\end{align}
%\textcolor{red}{Shouldn't we use a different letter instead of ``$t$'' in $\bar{X}(t)$? Like $\bar{X}(\tau)=...,\tau\in I_n$}
where $I_n := [t(n), t(n+1))$, and the time sequence ${t(n)}$ is given by $t(0) = 0$ and $t(n) = \sum_{k=0}^{n-1} \alpha_k$.
We next recall a standard result from stochastic approximation theory~\cite{doi:10.1137/S0363012904439301}, which plays a pivotal role in establishing the asymptotic convergence of the iterates.

\begin{proposition}
Suppose Assumptions \ref{ISS},~\ref{step-size}, \ref{Mar}, and \ref{noise} hold, and further assume that ${x_n}$ is bounded almost surely. Then, the continuous-time interpolation $\Bar{X}(t)$ of ${x_n}$ is an asymptotic pseudo-trajectory (APT) of the differential inclusion
\begin{equation}
\dot{x}(t) \in H(x) + \bar{B}(0,\epsilon).
\label{Perturbed}
\end{equation}
That is,
\begin{equation}
\lim_{t \to \infty}  \sup_{0 \leq s \leq T}
\norm{\Bar{X}(t+s) - x_t(s)} = 0,
\label{AST1}
\end{equation}
where $x_t(s)$ denotes the solution of~\eqref{Perturbed} with the initial condition $x_t(0) = \Bar{X}(t)$.
\label{Asympototic Pes}
\end{proposition}

In the ensuing Theorem, we prove that the iterates $x_n$ almost surely converges  to a neighborhood of $0$.

\begin{theorem}[Almost Sure Convergence]
\label{thm:eps_convergence}
Let Assumptions \ref{ISS},~\ref{step-size}, \ref{Mar}, and \ref{noise} hold, and further assume that the sequence $\{x_n\}$ generated by \eqref{eq:sri} is bounded almost surely.
Then, for any $\delta>0$, there exists $\epsilon_0>0$ such that for every $0<\epsilon\le\epsilon_0$ there exists $n_0\in\mathbb{N}$ with
\[
\|x_n\|\le\delta \qquad \forall\, n\ge n_0.
\]
\end{theorem}
\begin{proof}
We first show that there exists $T_0>0$ such that
\[
\|\Bar{X}(t)\|\le\delta \qquad \forall\, t\ge T_0,
\]
which, by the definition of the continuous-time interpolation $\Bar{X}(t)$, implies the desired bound for the discrete iterates $\{x_n\}$. Fix $\delta>0$ and choose a horizon $T>0$ (to be specified later). By Proposition~\ref{Asympototic Pes}, $\Bar{X}(t)$ is an asymptotic pseudo-trajectory of the perturbed differential inclusion~\eqref{Perturbed}. Hence, there exists $t_1=t_1(T,\delta)$ such that for all $t\ge t_1$,
\begin{equation}
\sup_{0\le s\le T}\,\|\Bar{X}(t+s)-x_t(s)\| \le \frac{\delta}{3}.
\label{apt}
\end{equation}

By the ISS assumption, there exist $\beta\in\mathcal{KL}$ and $\gamma\in\mathcal{K}_\infty$ such that every solution of~\eqref{Perturbed} satisfies
\[
\|x_t(s)\| \le \beta(\|\Bar{X}(t)\|,s)+\gamma(\epsilon) \qquad \forall\, s\ge 0.
\]
Because $\gamma\in\mathcal{K}_\infty$, there exists $\epsilon_0>0$ with
\[
0<\epsilon\le\epsilon_0 \;\Rightarrow\; \gamma(\epsilon)\le \frac{\delta}{3}.
\]
Since $\{x_n\}$ is almost surely bounded, the interpolation $\Bar{X}(t)$ takes values in some compact set $A\subset\mathbb{R}^d$ (possibly sample-path dependent). Fix one such sample path and the corresponding compact set $A$. We claim that there exists $T>0$ such that
\[
\beta(\|\xi\|,T)\le \frac{\delta}{3} \qquad \forall\, \xi\in A.
\]
Indeed, for each $\xi\in A$, the $\mathcal{KL}$ property gives $T_\xi>0$ with
\[
\beta(\|\xi\|,t)\le \frac{\delta}{6} \qquad \forall\, t\ge T_\xi.
\]
By continuity of $\beta(\cdot,t)$ in the first argument, there exists an open neighborhood $U_\xi$ of $\xi$ such that
\[
\beta(\|y\|,t)\le \frac{\delta}{6} \qquad \forall\, y\in U_\xi,\ \forall\, t\ge T_\xi.
\]
The family $\{U_\xi:\xi\in A\}$ covers $A$, hence admits a finite subcover $U_{\xi_1},\dots,U_{\xi_N}$. Define
\[
T = \max_{1\le i\le N} T_{\xi_i}.
\]
Then $\beta(\|\xi\|,t)\le \delta/3$ for all $\xi\in A$ and $t\ge T$. Now fix this $T$ and let $t\ge t_1$. Using the triangle inequality and in view of \eqref{apt}, we obtain
\begin{align}
\|\Bar{X}(t+T)\|
&\le \|\Bar{X}(t+T)-x_t(T)\| + \|x_t(T)\| \nonumber\\
&\le \frac{\delta}{3} + \underbrace{\beta(\|\Bar{X}(t)\|,T)}_{(a)} + \underbrace{\gamma(\epsilon)}_{(b)}. \label{eq:triangle}
\end{align}
By construction, $(a)\le \delta/3$ for all $\Bar{X}(t)\in A$, and $(b)\le \delta/3$ whenever $0<\epsilon\le\epsilon_0$. Therefore, from~\eqref{eq:triangle},
\[
\|\Bar{X}(t+T)\|\le \delta \qquad \forall\, t\ge t_1.
\]
Set $T_0 := t_1+T$. Then $\|\Bar{X}(t)\|\le \delta$ for all $t\ge T_0$. Finally, since $\Bar{X}(t)$ is a piecewise-linear interpolation of $\{x_n\}$ on each interval $[t(n),t(n+1))$, there exists $n_0\in\mathbb{N}$ such that $t(n_0)\ge T_0$, and hence
\[
\|x_n\|\le \delta \qquad \forall\, n\ge n_0.
\]
\end{proof}
\begin{remark}

%\textcolor{red}{
%In the above theorem, we have implicitly assumed that the origin is an equilibrium point of the differential inclusion~\eqref{Perturbed}. However, the same conclusion remains valid for any equilibrium point $x^\ast \in \mathbb{R}^d$ by considering the translated dynamics in terms of $x - x^\ast$. } \textcolor{red}{(Is this needed? is this not a trivial, as a result of coordinate change?)}

In the above Theorem, we assume that the sequence $\{x_n\}$ is bounded almost surely, which is a nontrivial requirement. Later, we show that this boundedness assumption holds naturally for a broad class of stochastic optimization problems. Nevertheless, if this condition does not hold, one standard approach to ensure boundedness is to project each iterate onto a compact, convex set. Such a projection preserves the asymptotic properties of the recursion while guaranteeing that $\{x_n\}$ remains bounded.
\end{remark}

In the next section, we provide sufficient conditions for the iterates $x_n$ to be almost surely bounded.

\section{Almost Sure Boundedness}

In this section, we consider the case where the set-valued map $H$ in~\eqref{eq:sri} reduces to a singleton and is globally Lipschitz continuous. In other words,
\begin{equation}
    x_{n+1} - x_n - \alpha_n M_{n+1} 
    \in \alpha_n \big(h(x_n) + \bar{B}(0,\epsilon) + M_{n+1}\big),
    \label{sriiss}
\end{equation}
where $h: \mathbb{R}^d \to \mathbb{R}^d$ is a single-valued mapping. We impose the following additional assumptions.

\begin{assumption}
\label{assump:Lip}
The map $h: \mathbb{R}^d \to \mathbb{R}^d$ is globally Lipschitz continuous; that is, there exists $L > 0$ such that for all $x, y \in \mathbb{R}^d$,
\[
\norm{h(x) - h(y)} \leq L \norm{x - y}.
\]
\end{assumption}

In addition, we strengthen Assumption~\ref{ISS} by requiring quadratic bounds on
the Lyapunov function. 

\begin{assumption}
\label{assump:V_equiv}
 For all $x \in \mathbb{R}^d$, there  constants $\underline{a}$ and $\overline{a} > 0$ such that
\[
\underline{a} \norm{x}^2 \leq V(x) \leq \overline{a} \norm{x}^2.
\]
\end{assumption}

Recall that, in view of Assumption~\ref{ISS}, we have
\[
\langle \nabla V(x),\, h(x) + b \rangle \leq -a(\norm{x}) + b(\epsilon),
\qquad \text{where } b \in \bar{B}(0, \epsilon).
\]
From the definition of a class~$\mathcal{K}_\infty$ function, we note that for any $\eta > 0$, there exists $R > 0$ such that for all $\norm{x} > R$,
\begin{equation}
    \langle \nabla V(x),\, h(x) + b \rangle \leq -\eta.
    \label{eta}
\end{equation}

Assumptions~\ref{assump:Lip} and~\ref{assump:V_equiv} are standard in the stochastic
approximation literature. In particular, under essentially the same assumptions,
almost sure boundedness of the recursion~\eqref{sriiss} has been established in
\cite{vidyasagar2024convergence,karandikar2024convergence} for the unbiased case,
that is, when the bias term is identically zero.

The proof of the almost sure boundedness result relies on the following
proposition, whose proof is provided in the Appendix \ref{appendix: prop2}.

\begin{proposition}
\label{prop:finite_horizon_bound}
Let $n \in \mathbb{N}$ be such that $\|x_n\| \le D$, and let $T > 0$ be arbitrary.  
Choose $m \in \mathbb{N}$ such that
\[
t(n+m) := \min \{\, k \in \mathbb{N} \mid t(n+k) \ge t(n) + T \,\}.
\]
Let $x(t)$ denote the solution of the differential inclusion
\begin{equation}
    \dot{x}(t) \in h(x(t)) + \bar{B}(0, \epsilon),
    \label{eq:ode}
\end{equation}
with initial condition $x(t(n)) = x_n$.  
Then, the continuous-time interpolation $\Bar{X}(t)$ of the iterates satisfies
\[
\sup_{0 \le i \le m} 
\|\Bar{X}(t(n+i)) - x(t(n+i))\|
\le K_{n,T}\, e^{L(T+1)},
\]
where
\begin{align*}
    K_{n,T} &:= 
L C_T \sum_{k=0}^{m-1} \alpha_{n+k}^2 
+ 2 \epsilon (T+1) 
+ \|\psi_{n,\, n+m}\|,\\
    \psi_{n,n+m} &:= \sum\limits_{k=0}^{m} \alpha_{n+k} M_{n+k+1},\\
    C_T &:= L \big( D + (T+1)\epsilon \big)\, e^{L(T+1)} + \epsilon.
\end{align*}
\end{proposition}

We are now in a position to establish the almost sure boundedness of the iterates
generated by~\eqref{sriiss}. We begin by showing that there exists a compact set
$C \subset \mathbb{R}^d$ such that the sequence $\{x_n\}$ is recurrent to $C$; that
is, almost surely, the iterates $\{x_n\}$ visit $C$ infinitely often.

\begin{lemma}
\label{lem:recurrence_compact}
The sequence of iterates $\{x_n\}$ generated by~\eqref{sriiss} is almost surely recurrent to a compact set $C \subset \mathbb{R}^d$, that is,
\[
\mathbb{P}\big( x_n \in C \text{ infinitely often} \big) = 1.
\]
\end{lemma}

\begin{proof}
    Since \(V\) is continuously differentiable with Lipschitz continuous gradient, we have
\[
V(x_{n+1}) 
\le 
V(x_n) 
+ \langle \nabla V(x_n),\, x_{n+1} - x_n \rangle 
+ \frac{L}{2} \|x_{n+1} - x_n\|^2.
\]
Taking conditional expectation with respect to \(\mathcal{F}_n\) and using the recursive inclusion, we obtain
\begin{equation}
\begin{split}
\mathbb{E}[V(x_{n+1}) \mid \mathcal{F}_n]
&\le 
V(x_n) 
+ \alpha_n \langle \nabla V(x_n),\, h(x_n) + b \rangle \\
& + \frac{3L}{2}\, \alpha_n^2\, 
\mathbb{E}\big[\|h(x_n)\|^2 + \|b\|^2 + \|M_{n+1}\|^2 \,\big|\, \mathcal{F}_n\big].
\end{split}
\label{eq:V_expansion}
\end{equation}
In view of Assumption~\ref{ISS}, we have
\[
\langle \nabla V(x_n),\, h(x_n) + b \rangle 
\le -a(\|x_n\|) + b(\epsilon).
\]
Furthermore, since \(h(\cdot)\) is globally Lipschitz with constant \(L\), we obtain
\[
\|h(x_n)\|^2 \le L^2 \|x_n\|^2 
\le \frac{L }{2 \Bar{a}}\, V(x_n),
\]
Substituting these bounds into~\eqref{eq:V_expansion} and letting $K_1 = K + \epsilon^2$ gives
\begin{equation}
\begin{split}
\mathbb{E}[V(x_{n+1}) \mid \mathcal{F}_n]
&\le 
\big(1 + \tfrac{3L^2 }{\Bar{a}}\, \alpha_n^2 \big) V(x_n) 
\\ & + \alpha_n \big(-a(\|x_n\|) + b(\epsilon)\big)
+ \tfrac{3L}{2}\, \alpha_n^2 K_1.
\end{split}
\label{eq:V_final}
\end{equation}
 Consider $\|x_n\| > R$ such that $-a(\|x_n\|) + b(\epsilon) < 0.$ Hence, whenever $\|x_n\| > R$, we have
\begin{equation}
\begin{split}
\mathbb{E}[V(x_{n+1}) \mid \mathcal{F}_n]
&\le 
\Big(1 + \tfrac{3L^2 a}{2}\, \alpha_n^2 \Big) V(x_n)
+ \tfrac{3L}{2}\, \alpha_n^2 K_1.
\end{split}
\label{eq:V_outside_R}
\end{equation}

Now consider two cases.

\smallskip
\noindent\textbf{Case 1.}  
Suppose there exists $n_0 \in \mathbb{N}$ such that 
$\|x_n\| > R$ for all $n \ge n_0$.  
Then inequality~\eqref{eq:V_outside_R} holds for all $n \ge n_0$, and the sequence $\{V(x_n)\}$ satisfies
\[
\mathbb{E}[V(x_{n+1}) \mid \mathcal{F}_n]
\le (1 + c_1 \alpha_n^2) V(x_n) + c_2 \alpha_n^2,
\]
for some constants $c_1, c_2 > 0$.  
By the Robbins–Siegmund theorem~\cite{ROBBINS1971233}, it follows that $\{V(x_n)\}$ converges almost surely and is therefore bounded.  
Consequently, there exists $R_1 > 0$ (possibly sample-path dependent) such that
\[
V(x_n) \le R_1, \qquad \forall\, n.
\]

\smallskip
\noindent\textbf{Case 2.}   
If the above condition does not hold, then there exist infinitely many indices $n$ such that $\|x_n\| \le R$.  
In this case, the compact set
\[
C := \{\, x \in \mathbb{R}^d \mid \|x\| \le R \,\}
\]
is recurrent; that is, the sequence $\{x_n\}$ returns to $C$ infinitely often almost surely.
 
Combining the two cases, we obtain that the sequence $\{x_n\}$ is almost surely recurrent to a compact set (either $\{\|x\| \le R_1\}$ or $\{\|x\| \le R\}$, depending on the trajectory).

\end{proof}

\begin{theorem}
\label{thm:boundedness}
Suppose Assumptions~\ref{ISS}, \ref{step-size}, \ref{Mar},   \ref{noise} and \ref{assump:Lip} hold.  
Then the sequence of iterates $\{x_n\}$ generated by the stochastic recursive inclusion \eqref{sriiss}
is almost surely bounded, that is,
\[
\sup_{n \ge 0} \|x_n\| < \infty \quad \text{a.s.}
\]
\end{theorem}

\begin{proof}
Assume, by way of contradiction, that there exists a sample path along which the sequence $\{x_n\}$ is unbounded, and fix such a sample path.  

Let $R_1 > 0$ be sufficiently large, and consider the two compact sets $C_1 := \{\, x \in \mathbb{R}^d \mid V(x) \le R_1 \,\},\;
C_2 := \{\, x \in \mathbb{R}^d \mid V(x) \le R_1 + T\eta \,\}$
where $\eta$ is chosen according to~\eqref{eta}. We further select $R_1$ large enough so that the set $C_2$	
  is recurrent. The existence of such an 
$R_1 > 0$ follows directly from Lemma~\ref{lem:recurrence_compact}.

Select $\delta > 0$ such that $N_\delta(C_1) \subset C_2$, which is always possible by increasing $T$. Next, choose $\epsilon > 0$ satisfying
\[
2 \epsilon (T+1) e^{L(T+1)} \le \frac{\delta}{2}.
\]
Let $D$ denote the diameter of the set $C_2$, and choose a sufficiently large index $n_0$ such that for all $n \ge n_0$,
\[
\big(L C_T \sum_{k=0}^{m-1} \alpha_{n+k}^2 
+ 2 \epsilon (T+1) 
+ \|\psi_{n,\, n+m}\|\big) e^{L(T+1)} \le \delta,
\]
where $C_T$ and $\psi_{n,n+m}$ are defined in Proposition~\ref{prop:finite_horizon_bound}.  
Such a choice of $n_0$ exists due to Assumption~\ref{step-size} and Lemma~\ref{lemma:noise_limit}.  

From Lemma~\ref{lem:recurrence_compact}, if $\{x_n\}$ is unbounded along this sample path, then it must be recurrent to the set $C_2$.  
Hence, there exists $n \ge n_0$ such that $x_n \in C_2$.  

Now, let $m \in \mathbb{N}$ be chosen such that
\[
t(n+m) := \min \{\, k \in \mathbb{N} \mid t(n+k) \ge t(n) + T \,\}.
\]
In view of Proposition~\ref{prop:finite_horizon_bound}, we have
\[
\sup_{0 \le i \le m} 
\|\Bar{X}(t(n+i)) - x(t(n+i))\| \le \delta,
\]
where $x(t(\cdot))$ denotes the solution of the differential inclusion~\eqref{eq:ode} with initial condition $x(t(n)) = \Bar{X}(t(n))$.

By construction of $C_2$ and from~\eqref{eta}, it follows that $x(t) \in C_1$ for all $t \ge t(n)$.  
Furthermore, by the definition of $T$ and $t(n+m)$, we have $x(t(n+m)) \in C_1$.  
Hence, for all $0 \le i < m$, we have
\[
\Bar{X}(t(n+i)) \in N_\delta(C_2),
\qquad
\text{and} \qquad
\Bar{X}(t(n+m)) \in N_\delta(C_1) \subset C_2.
\]

Therefore, $\Bar{X}(t(n+m)) \in C_2$.  
Repeating this argument iteratively shows that $\Bar{X}(t) \in N_\delta(C_2)$ for all $t \ge t(n+m)$.  
This contradicts the assumption that $\{x_n\}$ is unbounded.  
Hence, the sequence $\{x_n\}$ is almost surely bounded.
\end{proof}

\section{Leveraging Theorem \ref{thm:eps_convergence} for Tackling Non-Convex Unconstrained Challenges}

In this section, we study the stochastic optimization problem
\[
    \min_{x \in \mathbb{R}^d} f(x)
    :=
    \mathbb{E}[F(x,\zeta)],
\]
where $\zeta$ is a random variable defined on the probability space
$(\Omega,\mathcal{F},\mathbb{P})$ and $F : \mathbb{R}^d \times \Omega \to \mathbb{R}$ is smooth for every realization $\zeta$.
We consider stochastic gradient descent with a biased gradient oracle.
Specifically, at each iterate $x_n$, the oracle returns a stochastic gradient
$\widetilde{\nabla} f(x_n)$ whose bias and variance are controlled as
\begin{equation}
    \Big\|
    \mathbb{E}\big[\widetilde{\nabla} f(x_n)\mid\mathcal{F}_n\big]
    -
    \nabla f(x_n)
    \Big\|
    \le
    \frac{b_1}{\lambda}
    +
    b_2 \lambda,
    \label{bias}
\end{equation}
and
\begin{equation}
    \mathbb{E}\!\left[
    \big\|\widetilde{\nabla} f(x_n)-\nabla f(x_n)\big\|^2
    \,\middle|\,
    \mathcal{F}_n
    \right]
    \le
    \frac{b_3}{\lambda^2},
    \label{variance}
\end{equation}
where $b_1,b_2,b_3>0$ and $\lambda>0$ is a tunable parameter.
Such biased oracles naturally arise in zeroth-order optimization, where gradients
are approximated using noisy function evaluations, please see
\cite{prashanth2025gradient,10143924,paul2025convergence}.
In this section, we impose the following additional assumptions.

\begin{assumption}
\label{ass:L}
The gradient of $f$ is $L$-Lipschitz continuous, that is,
\[
    \|\nabla f(x) - \nabla f(y)\|
    \le
    L \|x - y\|,
    \qquad
    \forall\, x,y \in \mathbb{R}^d.
\]
\end{assumption}

\begin{assumption}
\label{ass:PL}
The function $f$ satisfies the Polyak--Łojasiewicz (PL) inequality: there exists
$\mu > 0$ such that
\[
    \frac{1}{2}\|\nabla f(x)\|^2
    \ge
    \mu \big(f(x) - f^\ast\big),
    \qquad
    \forall\, x \in \mathbb{R}^d,
\]
where $f^\ast := \inf_x f(x)$.
\end{assumption}

\begin{assumption}
\label{ass:QUB}
For all $x\in\mathbb{R}^d,\;\exists\;r_1, r_2 > 0$ such that
\[
    r_1 \|x - x^\ast\|^2
    \le
    f(x) - f(x^\ast)
    \le
    r_2 \|x - x^\ast\|^2,
    \qquad
    \forall\, x \in \mathbb{R}^d,
\]
where $x^\ast$ denotes the unique minimizer of $f$.
\end{assumption}

Assumptions~\ref{ass:L}--\ref{ass:QUB} are standard in the analysis of nonconvex
stochastic gradient methods; please see
\cite{vidyasagar2024convergence,karandikar2024convergence} for further discussion.

%Throughout this section, we impose Assumptions~\ref{ass:L}--\ref{ass:QUB}, which are
%standard in the analysis of nonconvex stochastic gradient methods
%\cite{vidyasagar2024convergence,karandikar2024convergence}.
%In particular, Assumption~\ref{ass:L} implies that $f$ is $L$-smooth and satisfies
%\[
 %   f(y)
   % \le
  %  f(x)
    %+
    %\langle \nabla f(x), y-x \rangle
    %+
    %\frac{L}{2}\|y-x\|^2,
    %\qquad \forall\, x,y \in \mathbb{R}^d.
%\]
%Choosing $y = x - \frac{1}{L}\nabla f(x)$ yields
%\begin{equation}
   % \frac{1}{2}\|\nabla f(x)\|^2
    %\le
    %L\big(f(x) - f^\ast\big).
    %\label{L-smoothness}
%\end{equation}

The stochastic gradient descent update with the biased oracle is given by
\begin{equation}
    x_{n+1}
    =
    x_n
    -
    \alpha_n \widetilde{\nabla} f(x_n).
    \label{SGD-ZO}
\end{equation}
Using \eqref{bias}--\eqref{variance}, the recursion \eqref{SGD-ZO} can be expressed
in the perturbed form
\[
    x_{n+1} - x_n - \alpha_n M_{n+1}
    \;\in\;
    \alpha_n\big(-\nabla f(x_n) + \bar{B}(0,\epsilon)\big),
\]
where $\{M_{n+1}\}$ is a martingale difference sequence and
\[
    \epsilon := \frac{b_1}{\lambda} + b_2 \lambda.
\]

The perturbation satisfies Assumptions~\ref{ISS}, \ref{step-size}, \ref{Mar},  and
\ref{noise} required by Theorem~\ref{thm:eps_convergence}.
Consequently, whenever the iterates $\{x_n\}$ are almost surely bounded, the
asymptotic behavior of \eqref{SGD-ZO} is governed by the perturbed gradient flow
\begin{equation}
    \dot x(t)
    \in
    -\nabla f\big(x(t)\big) + \bar{B}(0,\epsilon).
    \label{ISISI}
\end{equation}
By Proposition~1 in \cite{sontag2022remarks}, the differential inclusion
\eqref{ISISI} is input-to-state stable. Combining this with the almost sure
boundedness of the iterates (Theorem~\ref{thm:boundedness}) implies that
$\{x_n\}$ converges almost surely to a neighborhood of the minimizer $x^\ast$,
whose radius is controlled by $\epsilon$.

\section{Leveraging Theorem \ref{thm:eps_convergence} to Address Convex Constrained Optimization Challenges}

Assumption~\ref{ass:L} can be restrictive for several convex optimization problems.
Accordingly, in this section we do not rely on Lipschitz continuity of the gradient.
Instead, we work under the standard convexity framework.

Specifically, we consider the stochastic optimization problem
\begin{equation}
      \min_{x \in \mathcal{X}} f(x) := \mathbb{E}[F(x,\zeta)],
      \label{PSC}
\end{equation}
where, for every realization $\zeta$, the function $F(\cdot,\zeta)$ is convex in $x$,
and the feasible set $\mathcal{X} \subset \mathbb{R}^d$ is compact.

Under this setup, we assume access to a biased stochastic subgradient oracle.
At each iterate $x_n$, the oracle returns a stochastic subgradient
$\widetilde{g}(n)$ satisfying
\begin{equation}
    \mathbb{E}[\widetilde{g}(n)\mid \mathcal{F}_n]
    =
    g(n) + \mathrm{B}(n),
    \qquad
    g(n) \in \partial f(x_n),
    \label{Rumpaiusu}
\end{equation}
where the bias term $\mathrm{B}(n)$ is uniformly bounded as
\[
    \|\mathrm{B}(n)\|
    \le
    \frac{b_1}{\lambda} + b_2 \lambda.
\]
Moreover, the second moment of the stochastic subgradient satisfies
\begin{equation}
    \mathbb{E}\!\left[\|\widetilde{g}(n)\|^2 \mid \mathcal{F}_n\right]
    \le
    \frac{b_3}{\lambda^2},
    \label{variance-cvx}
\end{equation}
where $b_1, b_2, b_3 > 0$ are constants and $\lambda > 0$ is a controllable parameter.

Such biased stochastic subgradients naturally arise in zeroth-order stochastic
optimization, where subgradients are approximated using noisy function evaluations;
see \cite{10143924} for explicit constructions and expressions for the constants
$b_1$, $b_2$, and $b_3$.

To solve \eqref{PSC}, we employ the projected stochastic subgradient method,
where the true subgradient is replaced by the biased stochastic subgradient
$\widetilde{g}(n)$ defined in \eqref{Rumpaiusu}–\eqref{variance-cvx}.
Let $x_n$ denote the iterate at time $n$. The update rule is given by
\begin{equation}
    x_{n+1}
    =
    \mathcal{P}_{\mathcal{X}}
    \big(x_n - \alpha_n\, \widetilde{g}(n)\big),
    \label{PZOSGD-ZO1}
\end{equation}
where $\mathcal{P}_{\mathcal{X}}(\cdot)$ denotes the Euclidean projection onto
the feasible set $\mathcal{X}$.

In this section, we assume that the subdifferential mapping of $f$ is strongly
monotone, as stated below.

\begin{assumption}
\label{assum: monotonicity}
There exists a constant $M>0$ such that, for all $x,y \in \mathcal{X}$ and for all
$g(x)\in \partial f(x)$, $g(y)\in \partial f(y)$,
\[
    \langle g(x)-g(y),\, x-y \rangle
    \ge
    M \|x-y\|^2.
\]
\end{assumption}

In addition, we impose the following boundedness assumption on the stochastic
subgradients.

\begin{assumption}
\label{ass:bounded-subgrad}
There exists a constant $G>0$ such that
\[
    \|\widetilde{g}(n)\| \le G,
    \qquad \forall\, n.
\]
\end{assumption}

Assumption~\ref{ass:bounded-subgrad} is non-restrictive in the context of zeroth-order
optimization. When stochastic subgradients are obtained by randomly perturbing
the current iterate, the compactness of $\mathcal{X}$ implies that the objective
function is globally Lipschitz over $\mathcal{X}$, and hence the corresponding
subgradients are uniformly bounded. We refer the reader to
\cite{paul2025convergence} for a detailed discussion.

\begin{comment}
\begin{lemma}
For the zeroth-order gradient estimator $\widetilde{g}(n)$, we have
\[
\mathbb{E}[\widetilde{g}(n)\mid \mathcal{F}_n] = g(n) + \mathrm{B}(n); \quad g(n) \in \partial f(x_n)
\]
where the bias term satisfies
\[
\|\mathrm{B}(n)\| \leq L \lambda \sqrt{n} + K\, \frac{\mathcal{B}\sqrt{n}}{\lambda},
\]
for some constant $K > 0$. Moreover,
\[
\mathbb{E}\!\left[\|\widetilde{g}(n)\|^2 \mid \mathcal{F}_n\right]
    \leq K\, \frac{K_1^2}{\mu^2}\, n.
\]
\end{lemma}
\end{comment}

To express the iterate update in \eqref{PZOSGD-ZO1} as a stochastic recursive inclusion of the form in \eqref{eq:sri}, we invoke Proposition~5.3.5 from \cite{hiriart2013convex}. 
\begin{comment}
\begin{proposition}
\label{Prop:Proposition3}
For any convex set $\mathcal{X}$ and any $v \in \mathbb{R}^d$, we have
\[
\lim_{t \downarrow 0} \frac{\mathcal{P}_{\mathcal{X}}(x + t v) - x}{t}
    = \mathcal{P}_{\mathcal{T}_{\mathcal{X}}(x)}(v),
\]
where $\mathcal{T}_{\mathcal{X}}(x)$ denotes the tangent cone of the convex set $\mathcal{X}$ at the point $x$.
\end{proposition}
\end{comment}
Applying Proposition 5.3.5, we can rewrite the update rule in \eqref{PZOSGD-ZO1} as  
\[
x_{n+1}
= x_n + \alpha_n\, \mathcal{P}_{\mathcal{T}_{\mathcal{X}}(x_n)}\big(-\widetilde{g}(n)\big)
= x_n + \alpha_n\big(-\widetilde{g}(n) - \eta_n\big),
\]
where $\eta_n \in \mathcal{N}_{\mathcal{X}}(x_n)$.  
The final equality follows from Moreau’s Decomposition Theorem (Theorem 3.2.5 in \cite{hiriart2013convex}).  
From the same theorem, we obtain
\[
\eta_n = \mathcal{P}_{\mathcal{N}_{\mathcal{X}}(x_n)}\big(-\widetilde{g}(n)\big).
\]

Since $0 \in \mathcal{N}_{\mathcal{X}}(x_n)$, it follows that  
\[
\|\eta_n + \widetilde{g}(n)\| \leq \|\widetilde{g}(n)\|.
\]
Hence,
\[
\|\eta_n\|
\leq \|\eta_n + \widetilde{g}(n)\| + \|\widetilde{g}(n)\|
\leq 2\|\widetilde{g}(n)\|
\leq 2G,
\]
where the final bound uses Assumption \ref{ass:bounded-subgrad}. Thus, the recursion \eqref{PZOSGD-ZO1} can be equivalently expressed as the stochastic recursive inclusion
\begin{equation}
x_{n+1} - x_n - \alpha_n M_{n+1}
\in \alpha_n\big( \partial f(x_n) + \bar{B}(0,\epsilon) + \widehat{\mathcal{N}}_{\mathcal{X}}(x_n)\big),
\label{PZSGDISS}
\end{equation}
where
\[
\widehat{\mathcal{N}}_{\mathcal{X}}(x_n)
= \{\nu \in \mathcal{N}_{\mathcal{X}}(x_n) : \|\nu\| \leq G\}.
\]

It is straightforward to verify that this stochastic inclusion satisfies Assumption~\ref{Mar}.  
Therefore, by Theorem~\ref{thm:eps_convergence}, the asymptotic behavior of $\{x_n\}$ is characterized by the differential inclusion
\begin{equation}
\dot{x}(t)
\in -\big(\partial f(x) + \widehat{\mathcal{N}}_{\mathcal{X}}(x) + \bar{B}(0,\epsilon)\big),
\label{ODI}
\end{equation}
where $\epsilon \leq \frac{b_1}{\lambda} + b_2 \lambda$. The differential inclusion above, particularly the case $\epsilon = 0$ has been extensively studied in the literature; see, for instance, \cite{dupuis1993dynamical} for a detailed analysis of well-posedness, viability and stability.  
In this section, however, our focus is on establishing input-to-state stability (ISS) for the case of nonzero $\epsilon$.

\begin{proposition}
The dynamical system described by \eqref{ODI} is input-to-state stable (ISS).
\label{ISS4}
\end{proposition}

\begin{proof}
Consider the Lyapunov function
\[
V(x) = \frac{1}{2}\|x - x^\ast\|^2,
\]
where $x^\ast$ is the unique solution of \eqref{PSC}. 
Since $x^\ast$ is the global minimizer, the optimality condition gives
\begin{equation}
0 \in \partial f(x^\ast) + \mathcal{N}_{\mathcal{X}}(x^\ast).
\label{opt}
\end{equation}

Let $x(t)$ be any Carathéodory solution of \eqref{ODI} with initial condition $x_0 \in \mathcal{X}$.  
The time derivative of $V$ along the trajectory satisfies
\[
\dot{V}(t)
= \langle \nabla V(x(t)), -g(t) - \eta(t) + b(t)\rangle,
\]
where $g(t) \in \partial f(x(t))$, $\eta(t) \in \widehat{\mathcal{N}}_{\mathcal{X}}(x(t))$, and $b(t) \in \bar{B}(0,\epsilon)$.  
Thus,
{\small
\[
\begin{aligned}
& \langle x(t) - x^\ast, -g(t) - \eta(t) + b(t)\rangle
\\ &\overset{(a)}{\le}
\langle x(t) - x^\ast, -g(t) + g(x^\ast)\rangle
    - \langle x(t) - x^\ast, \eta(t)\rangle
    + \langle b(t), x(t)-x^\ast\rangle \\
&\overset{(b)}{\le}
- M\|x(t)-x^\ast\|^2
 + \frac{1}{2\alpha}\|x(t)-x^\ast\|^2
 + \frac{\alpha}{2}\|b(t)\|^2 \\
&\le
-(M - \tfrac{1}{2\alpha})\|x(t)-x^\ast\|^2
    + \frac{\alpha}{2}\|b(t)\|^2.
\end{aligned}
\]
}
Inequality (a) follows from \eqref{opt}, which ensures  
$\langle g(x^\ast), x - x^\ast\rangle \ge 0$ for all $g(x^\ast)\in \partial f(x^\ast)$.  
Inequality (b) follows from Assumption~\ref{assum: monotonicity} and the Young–Fenchel inequality.

By choosing $\alpha$ such that $M > \frac{1}{2\alpha}$, the coefficient of $\|x(t)-x^\ast\|^2$ becomes negative.  
Hence, the system satisfies the ISS Lyapunov condition, and therefore \eqref{ODI} is input-to-state stable.
\end{proof}

\section{Simulation Results}
In this section, we validate the results derived for the stochastic gradient descent algorithm \eqref{SGD-ZO}. Consider the functions $f_1,f_2:\mathbb{R}^2\to\mathbb{R}$, where $f_1(x)=x_1^2+x_2^2+\sin(x_2)$ and $f_2(x)=x_1^4-x_1^2+x_2^2+\sin(x_2)$. We can see that $f_1$ satisfies Assumptions \ref{ass:L}-\ref{ass:QUB} while $f_2$ does not. Furthermore, $f_1^*(0,-0.4503)=-0.231$ and $f_1^*(\pm\frac{1}{\sqrt{2}},-0.453)=-0.481$. The step-size $\alpha_n$ for the ZO-SGD \eqref{SGD-ZO} is set to $0.01/n^{0.6}$ and the algorithm is run using various values for $\lambda\in\{0.0005,0.05,0.1,1\}$. The initial condition is set to $x_0=(1,1)$. We assume that the zeroth-order oracle is noisy with $e(x_n+\lambda u_n)\in \mathcal{N}(5,1)$ and $e(x_n-\lambda u_n)\in \mathcal{N}(1,1)$. The gradient for both functions $f_1$ and $f_2$ are approximated using the symmetric finite-difference estimator
\[
\Tilde{\nabla} f(x_n)
=
\frac{\Hat{f}(x_n+\lambda u_n)-\Hat{f}(x_n-\lambda u_n)}{2\lambda} u_n,
\quad
\Hat{f}(x)=f(x)+e(x).
\]
A Taylor series argument shows that this estimator satisfies the biased-oracle assumptions~\eqref{bias}–\eqref{variance}. The numerical results are shown in Fig.~\ref{fig: muvarwithPL} and Fig.~\ref{fig: muvarwithoutPL}.

%We compare the convergence of the zeroth-order stochastic gradient descent algorithm \eqref{SGD-ZO}, for different values of $\lambda$ when the zeroth-order oracle is biased. We provide two examples involving non-convex functions where one of them satisfies Assumptions \ref{ass:L}-\ref{ass:QUB}, while the other does not. 

\begin{figure}
         \centering
            \includegraphics[scale=0.35]{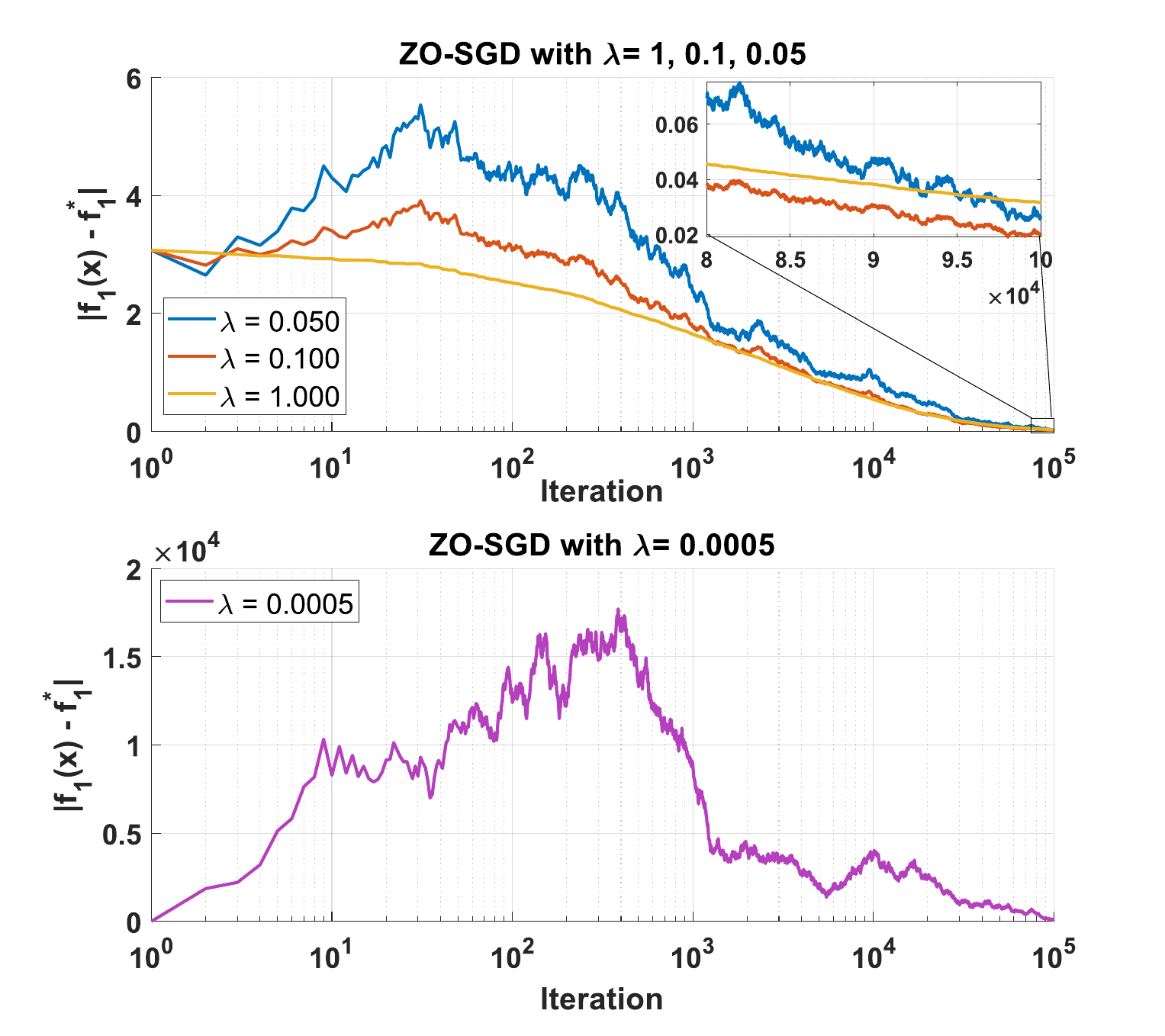}
         \caption{The top plot depicts $|f_1(x_n)-f_1^*|$, for $\lambda=0.05,\;0.1,\;1$ and the bottom plot depicts the same  for $\lambda=0.0005$. Moreover, $|f_1(x_{100000})-f_1^*|=73.4935, 0.0263, 0.0201  \;\mathrm{and}\;0.0316$ for $\lambda=0.0005,0.05,0.1\;\mathrm{and}\;1$, respectively.} 
         \label{fig: muvarwithPL}
     \end{figure}

\begin{figure}
         \centering
            \includegraphics[scale=0.35]{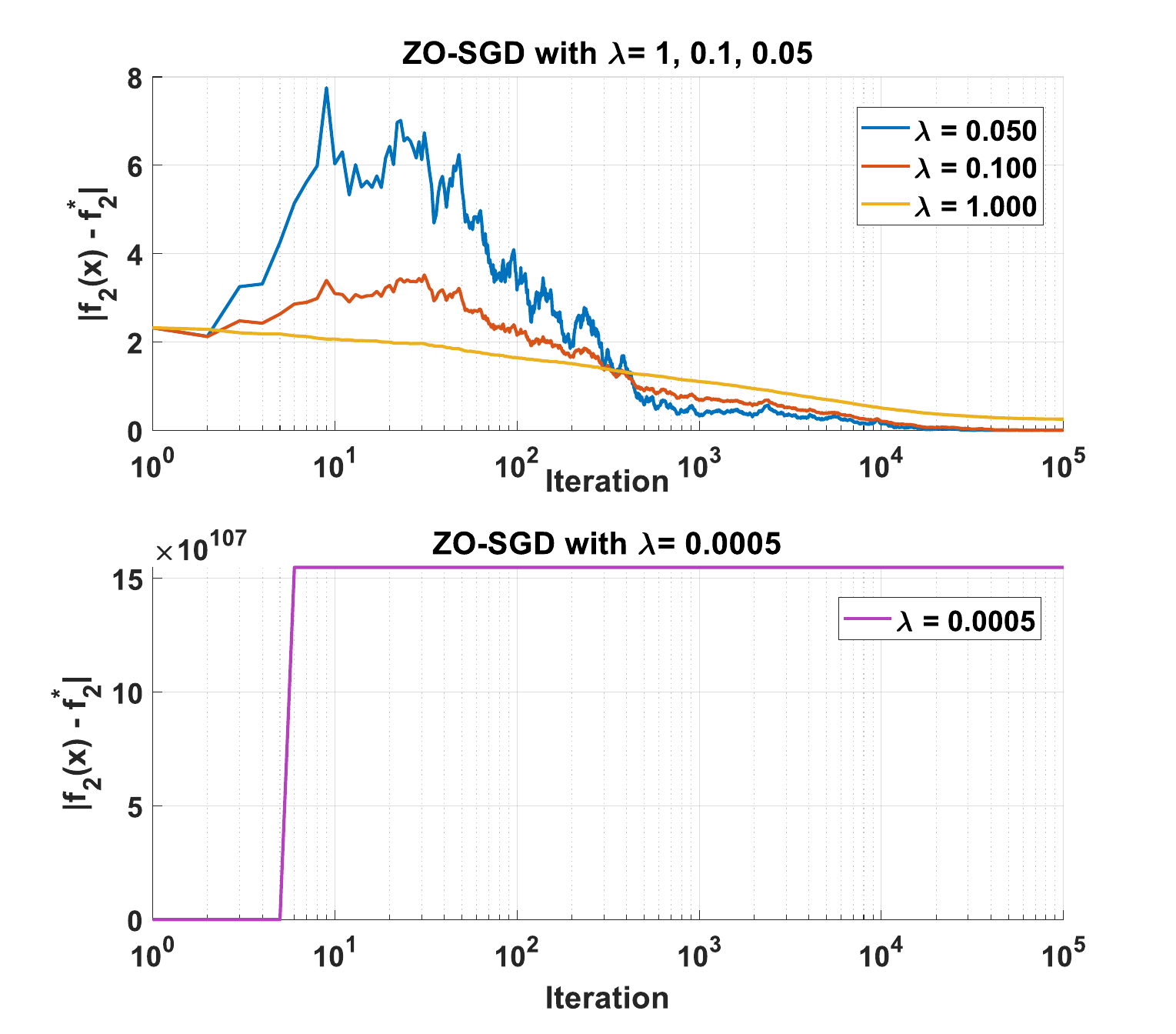}
         \caption{The top plot depicts $|f_2(x_n)-f_2^*|$, for $\lambda=0.05,\;0.1,\;1$ and the bottom plot depicts the same for $\lambda=0.0005$. Moreover, $|f_2(x_{100000})-f_2^*|=0.0064,\; 0.0068,\;    0.2545$ for $\lambda=0.05,\;0.1\;\mathrm{and}\;1$, respectively.} 
         \label{fig: muvarwithoutPL}
     \end{figure} 

We note that for $f_1(x)$,  meets all the assumptions given in this paper whereas these assumptions are not met for $f_2(x)$. 
From~\eqref{bias}, we observe that there exists a critical value
\(\lambda^\ast = \sqrt{b_1/b_2}\).
As \(\lambda\) is decreased toward \(\lambda^\ast\), the size of the
\(\delta\)-neighborhood to which the iterates \(x_n\) converge decreases,
in accordance with Theorem~\ref{thm:eps_convergence}.
However, reducing \(\lambda\) below \(\lambda^\ast\) increases the bias term,
leading to a larger limiting \(\delta\)-neighborhood.
Consequently, \(\lambda\) cannot be chosen arbitrarily small.

An explicit expression for the optimal value \(\lambda^\ast\) can be obtained
once the constants \(b_1\) and \(b_2\) are known. These constants can be derived from the
properties of the objective function, such as \(L\)-smoothness, and on the
statistical properties of the noise terms \(e(x_n \pm \lambda u_n)\); see~\cite{prashanth2025gradient}
for an exact expression.

Moreover, as indicated by~\eqref{variance}, the variance increases as
\(\lambda\) decreases. As a result, very small values of \(\lambda\) lead to
large oscillations of the iterates---see, for instance,
\(\lambda = 0.0005\) in Figs.~\ref{fig: muvarwithPL}
and~\ref{fig: muvarwithoutPL}---and consequently slower convergence to the
limiting neighborhood.

\section{Conclusion}

In this paper, we establish easily verifiable conditions for the almost sure boundedness and asymptotic convergence of stochastic recursive inclusions in the presence of a nonzero, non-diminishing bias. We also present representative examples that satisfy these conditions. Several directions for future research remain open, including non-asymptotic analysis of the iterates generated by \eqref{eq:sri}, relaxing the Lipschitz continuity assumptions while still guaranteeing almost sure boundedness, and a more detailed study of the role of the smoothing parameter 
$\lambda$ in stochastic zeroth-order optimization.
\appendices
\section{Proof of Proposition \ref{prop:finite_horizon_bound}} \label{appendix: prop2}

The proof of Proposition~\ref{prop:finite_horizon_bound} follows the same general strategy as Lemma~4.4 in~\cite{doi:10.1137/S0363012993253534} and the arguments presented in Chapter~3 of~\cite{borkar2008stochastic}, with appropriate modifications to account for the presence of an additional bias term. We begin by establishing the following two auxiliary lemmas.

\begin{lemma}
Let $x(t)$ be a Carathéodory solution of the differential inclusion
\[
\dot{x}(t) \in h(x(t)) + \bar{B}(0,\epsilon),
\]
with initial condition $x(0) = x_0$.
Then, for any $T>0$ and any $z \le T+1$, the following bound holds:
\[
\|x(z)\| 
\le \big( \|x_0\| + (T+1)\epsilon \big)\, e^{L(T+1)}.
\]
\label{Lemma1}
\end{lemma}

\begin{proof} The result follows by a straightforward application of the discrete-time Gronwall inequality.

\begin{comment}
Let $x(t)$ be a Carathéodory solution of the differential inclusion.  
Then, for any $z \ge 0$,
\[
x(z) = x_0 + \int_{0}^{z} \big( h(x(s)) + b(s) \big)\, ds,
\]
where $b(s) \in \bar{B}(0,\epsilon)$ for all $s$.  
Taking norms and applying the triangle inequality, we obtain
\begin{equation*}
\begin{split}
    \|x(z)\| 
    &\le \|x_0\| 
        + \int_{0}^{z} \|h(x(s))\|\, ds 
        + \int_{0}^{z} \|b(s)\|\, ds \\
    &\le \|x_0\| + L \int_{0}^{z} \|x(s)\|\, ds + \epsilon (T+1),
\end{split}
\end{equation*}
where the last term follows from $\|b(s)\| \le \epsilon$ and $z \le T+1$.  
By applying Grönwall’s inequality on $[0, z]$, we obtain
\[
\|x(z)\| \le \big( \|x_0\| + \epsilon (T+1) \big) e^{L(T+1)},
\]
which proves the desired result.
\end{comment}
\end{proof}

\begin{lemma}
\label{lemma:noise_limit}
Under Assumption~\ref{noise}, we have
\[
\lim_{n \to \infty} \; \sup_{k \ge n} 
\left\| \sum_{m = n}^{k} \alpha_m M_{m+1} \right\| = 0,
\quad \text{almost surely.}
\]
\end{lemma}

\begin{proof}
    Define  $  S_n = \sum\limits_{k=1}^{n} \alpha(k) M_{k+1}.$
    
 The sequence ${S_n},n\geq 0,$ is a martingale that is bounded in $L^2$,
    \begin{equation*}
        \mathbb{E}[S_n^2] = \sum\limits_{k=1}^{n} \mathbb{E} \alpha(k)^2 [\norm{M_{k+1}}^2] \leq K \sum\limits_{k=1}^{\infty} \alpha(k)^2.
    \end{equation*}
   By the martingale convergence theorem, $\{S_n\}$ converges almost surely. Consequently, its tail sequence 
    \begin{equation*}
        T_n = \sum\limits_{k=n}^{\infty} \alpha(k) M_{k+1}, n\geq 0,
    \end{equation*}
    also converges to zero almost surely. Therefore, for any \(k \ge n\),
\begin{equation*}
    \sum_{m=n}^{k} \alpha_m M_{m+1}
    = T_n - \sum_{m=k+1}^{\infty} \alpha_m M_{m+1}.
\end{equation*}
Taking norms and applying the definition of \(T_m\), we obtain
\begin{equation*}
    \sup_{k \ge n} \Big\|\sum_{m=n}^{k} \alpha_m M_{m+1} \Big\|
    \le \|T_n\| + \sup_{m \ge n} \|T_m\|.
\end{equation*}
Finally, letting \(n \to \infty\) and noting that \(T_n \to 0\) almost surely yields
\begin{equation*}
    \lim_{n \to \infty} \sup_{k \ge n} \Big\|\sum_{m=n}^{k} \alpha_m M_{m+1} \Big\| = 0.
\end{equation*}
\end{proof}

We are now in a position to prove the Proposition \ref{prop:finite_horizon_bound}.

\begin{proof}[Proof of Proposition \ref{prop:finite_horizon_bound}]

From the stochastic recursive inclusion in~\eqref{sriiss} and the definition of the continuous-time interpolated trajectory, we obtain
\begin{align}
    \Bar{X}(t(n+m))
    &= \Bar{X}(t(n)) \nonumber\\
    &+ \sum_{k=0}^{m-1} \alpha_{n+k} 
    \big( \Bar{H}(\Bar{X}(t(n+k))) + M_{n+k+1} \big),
\label{eq7}
\end{align}
where, for simplicity, we denote $\Bar{H}(x) := h(x) + b, b \in \bar{B}(0, \epsilon).$ Let $x(t)$ denote any solution of the differential inclusion \eqref{eq:ode}
with the initial condition $x(t(n)) = x_n$.  
Then, for all $t \geq t(n)$, we have
\begin{align}
     x(t(n+m)) &= x_n + \int_{t(n)}^{t(n+m)} \Bar{H}(x(z))\, dz \nonumber\\
    &= x_n 
    + \int_{t(n)}^{t(n+m)} \big( \Bar{H}(x(z)) - \Bar{H}(x([z])) \big)\, dz\nonumber 
   \\ & + \sum_{k=0}^{m-1} \alpha_{n+k} \Bar{H}(x(t(n+k))),
\label{eq8}
\end{align}
where $[z] := \max \{\, t(k) \mid t(k) \leq z \,\}.$ Subtracting~\eqref{eq8} from~\eqref{eq7}, we obtain  
\begin{equation}
    \begin{split}
           & \norm{\Bar{X}(t(n+m)) - x(t(n+m))} 
           \\  \leq & \underbrace{\int_{t(n)}^{t(n+m)}  \norm{\Bar{H}(x(z)) - \Bar{H}(x([z])) }  dz}_{P_1}  + \underbrace{\norm{\sum\limits_{k=0}^{m-1} \alpha_{n+k} M_{n+k+1}}}_{P_2}
           \\ & + \underbrace{\sum\limits_{k=0}^{m-1} \alpha_{n+k} \norm{\Bar{H}(x(t(n+k)) - \Bar{H}(\Bar{X}(t(n+k))}}_{P_3}. 
    \end{split}
    \label{Mimoia}
\end{equation}

We now derive bounds on each of the terms $P_1$, $P_2$, and $P_3$ separately.

\subsubsection*{Bound on $\mathbf{P_1}$}

\begin{equation}
    \begin{split}
     &   \int_{t(n)}^{t(n+m)}  \norm{\Bar{H}(x(z)) - \Bar{H}(x([z])) }  dz
     \\ \leq &  \int_{t(n)}^{t(n+m)}  \norm{h(x(z)) - h (x([z])) } dz + \int_{t(n)}^{t(n+m)}  \norm{b} dz
     \\ \leq & L \int_{t(n)}^{t(n+m)}  \norm{x(z) - x([z]) } dz + \epsilon (T+1).
    \end{split}
    \label{12eq}
\end{equation}

Without loss of generality, we assume that $\alpha_n \leq 1$ for all $n$.  
Then, from the definition of $t(n+m)$, it follows that
\[
t(n+m) \leq t(n) + T + 1.
\]

To bound the first term on the right-hand side of~\eqref{12eq}, we use Lemma~\ref{Lemma1}.  
In particular, we have
\begin{equation}
\begin{split}
   & \int_{t(n)}^{t(n+m)} \|x(z) - x([z])\|\, dz
  \\  &= \sum_{k=0}^{m-1} \int_{t(n+k)}^{t(n+k+1)} 
        \|x(z) - x(t(n+k))\|\, dz \\
    &\le \sum_{k=0}^{m-1} \int_{t(n+k)}^{t(n+k+1)} 
        \int_{t(n+k)}^{z} \|\dot{x}(s)\|\, ds\, dz \\
    &\le \sum_{k=0}^{m-1} \int_{t(n+k)}^{t(n+k+1)} 
        \int_{t(n+k)}^{z} \big( \|h(x(s))\| + \|b(s)\| \big)\, ds\, dz \\
    &\le C_T \sum_{k=0}^{m-1} \alpha_{n+k}^2,
\end{split}
\label{eq:T1_final_bound}
\end{equation}
where the last inequality follows from Lemma~\ref{Lemma1}. Thus, from \eqref{12eq} we obtain 
\begin{align*}
    \int\limits_{t(n)}^{t(n+m)}  \norm{\Bar{H}(x(z)) - \Bar{H}(x([z])) }  dz \leq  C_T  L \sum\limits_{k=0}^{m-1} \alpha_{n+k}^2  + \epsilon(T+1).
\end{align*}
         
\subsubsection*{Bound on $\mathbf{P_2}$}

Recall that $\psi_{n,\, n+m} = \sum_{k=0}^{m-1} \alpha_{n+k} M_{n+k+1}$. In view of the Lemma~\ref{lemma:noise_limit}, it follows that for any fixed $m \in \mathbb{N}$,
\[
\lim_{n \to \infty} \psi_{n,\, n+m} = 0, \qquad \text{almost surely.}
\]

\subsubsection*{Bound on $\mathbf{P_3}$}

\begin{equation*}
    \begin{split}
       & \sum\limits_{k=0}^{m-1} a_{n+k} \norm{\Bar{H}(x(t(n+k)) - \Bar{H}(\Bar{X}(t(n+k))}
       \\ \leq & L \sum\limits_{k=0}^{m-1} \alpha_{n+k} \norm{x(t(n+k)) - \Bar{X}(x(t(n+k))} + \epsilon(T+1).
    \end{split}
\end{equation*}

The last inequality is true because of the Lipschitz continuity of $h$ and  the definition of $t(n+m)$. 

Thus, from~\eqref{Mimoia}, we obtain
\begin{equation}
\begin{split}
   & \|\Bar{X}(t(n+m)) - x(t(n+m))\| 
   \\  \le & C_T L \sum_{k=0}^{m-1} \alpha_{n+k}^2 
        + 2 \epsilon (T+1) 
        + \|\psi_{n,\, n+m}\| \\
    &\quad + L \sum_{k=0}^{m-1} \alpha_{n+k} 
        \|x(t(n+k)) - \Bar{X}(t(n+k))\|.
\end{split}
\label{eq:Mimoia_final}
\end{equation}

Finally, by applying the discrete-time Grönwall inequality to~\eqref{eq:Mimoia_final}, we obtain
\[
\sup_{0 \le i \le m} 
\|\Bar{X}(t(n+i)) - x(t(n+i))\|
\le K_{n,T}\, e^{L(T+1)}
\].
\end{proof}

\bibliographystyle{IEEEtran}
\bibliography{ref}

%---------------------------

\end{document}